\documentclass[12pt,reqno]{amsart}
\usepackage{fullpage,url,amssymb,enumerate,colonequals}
 \usepackage[all]{xy} 
\usepackage{mathrsfs} 

\usepackage{color}

\newcommand{\defi}[1]{\textsf{#1}} 

\newcommand{\C}{\mathbb{C}}
\newcommand{\rmE}{\operatorname{E}}
\newcommand{\F}{\mathbb{F}}
\newcommand{\G}{\mathbb{G}}

\newcommand{\PP}{\mathbb{P}}
\newcommand{\Q}{\mathbb{Q}}
\newcommand{\R}{\mathbb{R}}
\newcommand{\rmT}{\operatorname{T}}
\newcommand{\Z}{\mathbb{Z}}
\newcommand{\Qbar}{{\overline{\Q}}}
\newcommand{\Xsep}{{X^{\operatorname{sep}}}}
\newcommand{\Xbar}{{\overline{X}}}
\newcommand{\Ysep}{{Y^{\operatorname{sep}}}}

\newcommand{\kbar}{{\overline{k}}}

\newcommand{\Lbar}{{\overline{L}}}
\newcommand{\ksep}{{k^{\operatorname{sep}}}}

\newcommand{\calD}{\mathcal{D}}

\newcommand{\calS}{\mathcal{S}}

\newcommand{\calU}{\mathcal{U}}

\newcommand{\calX}{\mathcal{X}}
\newcommand{\calY}{\mathcal{Y}}
\newcommand{\calZ}{\mathcal{Z}}

\newcommand{\FF}{\mathscr{F}}
\newcommand{\GG}{\mathscr{G}}

\newcommand{\LL}{\mathscr{L}}
\newcommand{\MM}{\mathscr{M}}
\newcommand{\OO}{\mathscr{O}}

\DeclareMathOperator{\Aut}{Aut}

\DeclareMathOperator{\Char}{char}

\DeclareMathOperator{\cl}{cl}

\DeclareMathOperator{\Div}{Div}

\DeclareMathOperator{\EffDiv}{\bf EffDiv}

\DeclareMathOperator{\Frac}{Frac}

\DeclareMathOperator{\Gal}{Gal}
\DeclareMathOperator{\Gr}{Gr}
\DeclareMathOperator{\Hilb}{Hilb}
\DeclareMathOperator{\Hom}{Hom}

\DeclareMathOperator{\NS}{NS}
\DeclareMathOperator{\Num}{Num}

\DeclareMathOperator{\Pic}{Pic}
\DeclareMathOperator{\PIC}{\bf Pic}

\DeclareMathOperator{\Proj}{Proj}

\DeclareMathOperator{\rk}{rk}

\DeclareMathOperator{\Spec}{Spec}

\DeclareMathOperator{\Tate}{Tate}
\DeclareMathOperator{\tr}{tr}

\newcommand{\et}{{\textup{\'et}}}

\newcommand{\sep}{{\operatorname{sep}}}

\newcommand{\tH}{{\operatorname{th}}}
\newcommand{\tors}{{\operatorname{tors}}}

\newcommand{\Cech}{\v{C}ech}

\newcommand{\GL}{\operatorname{GL}}
\newcommand{\HH}{{\operatorname{H}}}

\newcommand{\directsum}{\oplus} 
\newcommand{\Directsum}{\bigoplus} 
\newcommand{\injects}{\hookrightarrow}
\newcommand{\intersect}{\cap} 
\newcommand{\isom}{\simeq}

\renewcommand{\setminus}{-} 
\newcommand{\surjects}{\twoheadrightarrow}
\newcommand{\tensor}{\otimes} 
\newcommand{\Union}{\bigcup} 

\newcommand{\isomto}{\overset{\sim}{\rightarrow}}

\newcommand{\To}{\longrightarrow}

\numberwithin{equation}{section}

\newtheorem{theorem}[equation]{Theorem}
\newtheorem{lemma}[equation]{Lemma}
\newtheorem{corollary}[equation]{Corollary}
\newtheorem{proposition}[equation]{Proposition}

\newtheorem*{conjectureT}{Conjecture~$\rmT^p(X,\ell)$}
\newtheorem*{conjectureE}{Conjecture~$\rmE^p(X,\ell)$}
\newtheorem{hypothesis}[equation]{Hypothesis}

\theoremstyle{definition}
\newtheorem{definition}[equation]{Definition}
\newtheorem{setup}[equation]{Setup}

\theoremstyle{remark}
\newtheorem{remark}[equation]{Remark}

\usepackage[
	backref,
	pdfauthor={Bjorn Poonen}, 
]{hyperref}
\usepackage[alphabetic,backrefs,lite]{amsrefs} 

\begin{document}

\title{Computing N\'eron--Severi groups and cycle class groups}
\author{Bjorn Poonen}
\thanks{B.P.\ was supported by the Guggenheim Foundation and National Science Foundation grants DMS-0841321 and DMS-1069236.}
\address{Department of Mathematics, Massachusetts Institute of Technology, Cambridge, MA 02139-4307, USA}
\email{poonen@math.mit.edu}
\urladdr{\url{http://math.mit.edu/~poonen/}}

\author{Damiano Testa}
\address{Mathematics Institute, University of Warwick, Coventry, CV4 7AL, United Kingdom}
\email{adomani@gmail.com}
\urladdr{\url{http://homepages.warwick.ac.uk/~maskal/zone}}

\author{Ronald van Luijk}
\address{Mathematisch Instituut, Universiteit Leiden, Postbus 9512, 2300 RA, Leiden, the Netherlands}
\email{rvl@math.leidenuniv.nl}
\urladdr{\url{http://www.math.leidenuniv.nl/~rvl/}}

\date{July 23, 2014}

\begin{abstract}
Assuming the Tate conjecture and the computability of \'etale cohomology
with finite coefficients, we give an algorithm 
that computes the N\'eron--Severi group of 
any smooth projective geometrically integral variety,
and also the rank of the group of numerical equivalence classes
of codimension~$p$ cycles for any $p$.
\end{abstract}

\maketitle

\section{Introduction}\label{S:introduction}

Let $k$ be a field, and let $\ksep$ be a separable closure.
Let $X$ be a smooth projective geometrically integral $k$-variety,
and let $\Xsep \colonequals X \times_k \ksep$.

If $k=\C$, then the Lefschetz $(1,1)$ theorem
identifies the N\'eron--Severi group $\NS X$ 
(see Section~\ref{S:notation} for definitions) 
with the subgroup of $\HH^2(X(\C),\Z)$
mapping into the subspace $\HH^{1,1}(X)$ of $\HH^2(X(\C),\C)$.
Analogously, if $k$ is a finitely generated field,
then the Tate conjecture 
describes $(\NS \Xsep) \tensor \Q_\ell$
in terms of the action of $\Gal(\ksep/k)$ on $\HH^2_{\et}(\Xsep,\Q_\ell(1))$,
for any prime $\ell \ne \Char k$.

Can such descriptions be transformed into algorithms for computing $\NS \Xsep$?
To make sense of this question,
we assume that $k$ is replaced by 
a finitely generated subfield over which $X$ is defined;
then $X$ and $k$ admit a finite description suitable 
for computer input (see Section~\ref{S:explicit}).
Using the Lefschetz $(1,1)$ theorem 
involves working over the uncountable field $\C$,
while using the Tate conjecture 
involves an action of an uncountable Galois group 
on a vector space over an uncountable field $\Q_\ell$,
so it is not clear a priori that either approach 
can be made into an algorithm.

In this paper, assuming only the ability to compute
the finite Galois modules $\HH^i_{\et}(\Xsep,\mu_{\ell^n})$ 
for each $i \le 2$ and $n$, 
we give an algorithm for computing $\NS \Xsep$
that terminates if and only if the Tate conjecture holds for $X$
(Remark~\ref{R:NS}).
Moreover, if $k$ is finite, 
then we can even avoid computing the Galois modules
$\HH^i_{\et}(\Xsep,\mu_{\ell^n})$, 
by instead using point-counting to compute the zeta function of $X$,
as is well known 
(Theorem~\ref{T:Num variant}\eqref{I:algorithm B}).
In any case, we give an algorithm to compute $\HH^i_{\et}(\Xsep,\mu_{\ell^n})$ 
for any variety in characteristic~$0$ (Theorem~\ref{T:compute etale in char 0}) 
and any variety that lifts to characteristic~$0$ 
(Corollary~\ref{C:compute etale by lifting});
also, after the first version of the present article was made available, 
Madore and Orgogozo announced an algorithm to compute it 
in general \cite{Madore-Orgogozo-preprint}*{Th\'eor\`eme~0.9}
(they work over an algebraically closed ground field,
but the cohomology groups are unchanged in passing from $\ksep$ to $\kbar$).

Combining our results with the truth of the Tate conjecture for K3 surfaces $X$
over finitely generated fields of characteristic not~$2$
(\cites{Nygaard1983,Nygaard-Ogus1985,Maulik-preprint,Charles2013,Madapusi-preprint})
yields an unconditional algorithm for computing $\NS \Xsep$ for 
all such K3 surfaces 
(Theorem~\ref{T:unconditional NS for K3}).
(See \cite{Tate1994}*{Section~5} and \cite{Andre1996b} for some other 
cases in which the Tate conjecture is known.)
We also provide an unconditional algorithm for computing the torsion
subgroup $(\NS \Xsep)_\tors$ for any $X$ over any finitely generated field~$k$ 
(Theorem~\ref{T:NS_tors}).

Finally, we prove also statements for cycles of higher codimension.
In particular, we describe a conditional algorithm that
computes the rank of the group $\Num^p \Xsep$
of codimension~$p$ cycles modulo numerical equivalence
(Theorem~\ref{T:Num}).

If $\ksep$ is replaced by an algebraic closure $\kbar$
in any of the results above,
the resulting analogue holds (Remarks \ref{R:Num Xbar} and~\ref{R:kbar}).

\section{Previous approaches}
\label{S:previous approaches}

Several techniques exist in the literature for obtaining information
on N\'eron--Severi groups:
\begin{itemize}
\item Lower bounds on the rank are often obtained 
by exhibiting divisors explicitly.
\item An initial upper bound is given by the second Betti number,
which is computable (see Proposition~\ref{P:betti}).
\item Over $\C$, Hodge theory provides the improved upper bound $h^{1,1}$,
which again is computable.
(Indeed, software exists for computing
all the Hodge numbers $h^{p,q} \colonequals \dim \HH^q(X,\Omega^p)$,
as a special case of computing cohomology of coherent sheaves
on projective varieties \cite{Vasconcelos1998}*{Appendix~C.3}.)
\item Over a finite field $k$, computation of the zeta function
can yield an improved upper bound: see Section~\ref{S:alternative}
for details.
\item Over finitely generated fields $k$, one can spread out $X$
to a smooth projective scheme $\calX$ over a finitely generated $\Z$-algebra
and reduce modulo maximal ideals to obtain 
injective specialization homomorphisms 
$(\NS \Xsep) \tensor \Q \to (\NS \calX_{\overline{F}}) \tensor \Q$ 
where $F$ is the 
finite residue field
(see \cite{VanLuijk2007-Heron}*{Proposition~6.2} or 
\cite{Maulik-Poonen2012}*{Proposition~3.6}, for example).
Combining this with the method of the previous item
bounds the rank of $\NS \Xsep$.
In some cases, one can prove directly that certain elements of 
$(\NS \calX_{\overline{F}}) \tensor \Q$ 
are not in the image of the specialization homomorphism, 
to improve the bound~\cite{Elsenhans-Jahnel2011-oneprime}.
\item The previous item can be improved also by using more than one reduction
if one takes into account that the specialization homomorphisms 
preserve additional structure, such as the intersection pairing 
in the case $\dim X=2$~\cite{VanLuijk2007}
or the Galois action~\cite{Elsenhans-Jahnel2011}.
In the $\dim X=2$ case, 
the discriminant of the intersection pairing 
can be obtained, up to a square factor,
either from explicit generators for 
$(\NS \calX_{\overline{F}}) \tensor \Q$~\cite{VanLuijk2007}
or from the Artin--Tate conjecture~\cite{Kloosterman2007}.
F.~Charles proved that for a K3 surface $X$ over a number field,
the information from reductions is sufficient to determine
the rank of $\NS \Xsep$, assuming the Hodge conjecture for 
$2$-cycles on $X \times X$~\cite{Charles2014}.
\item If $X$ is a quotient of another variety $Y$ by a finite group $G$,
then the natural map $(\NS \Xsep) \tensor \Q \to ((\NS \Ysep) \tensor \Q)^G$
is an isomorphism.
For instance, this has been applied to \defi{Delsarte surfaces},
i.e., surfaces in $\PP^3$ defined by a homogeneous form with
four monomials, using that they are quotients of 
Fermat surfaces~\cite{Shioda1986}.
\item When $X$ is an elliptic surface, the rank of $\NS \Xsep$
is related to the rank of the Mordell--Weil group of the generic fiber 
\citelist{\cite{Tate1995}*{p.~429}; \cite{Shioda1972}*{Corollary~1.5}; \cite{Shioda1990}*{Corollary~5.3}}.
This has been generalized in various ways, for example to 
fibrations into 
abelian varieties \citelist{\cite{Kahn2009}; \cite{Oguiso2009}*{Theorem~1.1}}.
\item When $X$ is a K3 surface of degree~$2$ over a number field,
the Kuga--Satake construction relates the Hodge classes on $X$
to the Hodge classes on an abelian variety of dimension $2^{19}$.
B.~Hassett, A.~Kresch, and Yu.~Tschinkel use this to give 
an algorithm to compute $\NS \Xsep$ 
for such $X$~\cite{Hassett-Kresch-Tschinkel2013}*{Proposition~19}.
\end{itemize}
Also, \cite{Simpson2008} shows that if one assumes the Hodge conjecture,
then one can decide, given a nice variety $X$ over $\Qbar \subseteq \C$
and a singular homology class $\gamma \in \HH_{2p}(X(\C),\Q)$,
whether $\gamma$ is the class of an algebraic cycle.

\section{Notation}
\label{S:notation}

Given a module $A$ over an integral domain $R$,
let $A_{\tors}$ be its torsion submodule,
let $\widetilde{A} \colonequals A/A_{\tors}$,
and let $\rk A \colonequals \dim_K(A \tensor_R K)$ 
where $K\colonequals \Frac R$.
If $A$ is a submodule of another $R$-module $B$,
the \defi{saturation} of $A$ in $B$ is 
$\{b \in B: nb \in A \textup{ for some nonzero $n \in R$}\}$.
If $A$ is a $G$-module for some group $G$,
then $A^G$ is the subgroup of invariant elements.
We say that a $G$-module $A$ is \defi{finite} (resp. \defi{finitely generated})
if it is so as a set (resp.\ abelian group).

Given a field $k$, let $\kbar$ be an algebraic closure, 
let $\ksep$ be the separable closure inside $\kbar$,
let $G_k \colonequals \Gal(\ksep/k) \isom \Aut(\kbar/k)$, 
and let $\kappa$ be the characteristic of $k$.
A \defi{variety} $X$ over a field $k$ 
is a separated scheme of finite type over $k$.
For such $X$, let $\Xsep \colonequals X \times_k \ksep$ 
and $\Xbar \colonequals X \times_k \kbar$.
Call $X$ \defi{nice} if it is smooth, projective, and geometrically integral.

Suppose that $X$ is a nice $k$-variety.
Let $\Pic X$ be its \defi{Picard group}.
Let $\PIC_{X/k}$ be the \defi{Picard scheme} of $X$ over $k$.
There is an injection $\Pic X \to \PIC_{X/k}(k)$, 
but it is not always surjective.
Let $\PIC^0_{X/k}$ be the connected component of the identity in $\PIC_{X/k}$.
Let $\Pic^0 X \le \Pic X$ be the group
of isomorphism classes of line bundles 
such that the corresponding $k$-point of $\PIC_{X/k}$
lies in $\PIC^0_{X/k}$;
any such line bundle $\LL$ (or divisor representing it) is called 
\defi{algebraically equivalent to $0$}.  Equivalently, a line bundle $\LL$ is 
algebraically equivalent to $0$ if there is a connected variety $B$ and a 
line bundle $\MM$ on $X \times B$
such that $\MM$ restricts to the trivial line bundle above one point of $B$
and to $\LL$ above another
(this holds even over the ground field $k$: take $B$ to be a component
$H$ of $\EffDiv_X$ lying above a translate of $\PIC^0_{X/k}$
as in Lemma~\ref{L:no functors}(\ref{I:nice bijection},\ref{I:A exists})).
Define the \defi{N\'eron--Severi group} $\NS X$ 
as the quotient $\Pic X/\Pic^0 X$;
it can be identified with the set of components of $\PIC_{X/k}$
containing the class of a divisor of $X$ over $k$
(which is stronger than assuming that the component has a $k$-point).
Then $\NS X$ is a finitely generated 
abelian group \cite{Neron1952}*{p.~145,~Th\'{e}or\`{e}me~2}
(see \cite{SGA6}*{XIII.5.1} for another proof).
Let $\PIC^\tau_{X/k}$ be the finite union of connected components of
$\PIC_{X/k}$ parametrizing classes of line bundles
whose class in $\NS \Xbar$ is torsion.

Let $\calZ^p(X)$ be the group of codimension~$p$ cycles on $X$.
Let $\Num^p X$ be the quotient of $\calZ^p(X)$
by the subgroup of cycles numerically equivalent to $0$.
Then $\Num^p X$ is a finite-rank free abelian group.
Let $\calZ^1(X)^\tau$ be the set of divisors $z \in \calZ^1(X)$
having a positive multiple that is algebraically equivalent to $0$.
Let $(\Pic X)^\tau$ be the image of $\calZ^1(X)^\tau$ 
under $\calZ^1(X) \to \Pic X$.

If $m \in \Z_{>0}$ and $\kappa \nmid m$, and $i,p \in \Z$, 
let $\HH^i(\Xsep,(\Z/m\Z)(p))$ be the \'etale cohomology group;
this is a finite abelian group. 
For each prime $\ell \ne \kappa$, define
$\HH^i(\Xsep,\Z_\ell(p))\colonequals 
\varprojlim_n \HH^i(\Xsep,(\Z/\ell^n\Z)(p))$,
a finitely generated $\Z_\ell$-module; 
and define $\HH^i(\Xsep,\Q_\ell(p)) \colonequals 
\HH^i(\Xsep,\Z_\ell(p)) \tensor_{\Z_\ell} \Q_\ell$,
a finite-dimensional $\Q_\ell$-vector space; 
its dimension $b_i(X)$ is independent of $p$,
and is called an \defi{$\ell$-adic Betti number}.

Let $X$ be a nice $k$-variety.
Let $K(X)$ be its Grothendieck group of coherent sheaves.
For a coherent sheaf $\FF$ on a projective variety $X$,
define $\chi(\FF) \colonequals \sum_{i \ge 0} (-1)^i \dim \HH^i(X,\FF)$;
this induces a homomorphism $\chi \colon K(X) \to \Z$
sending the class $\cl(\FF)$ of $\FF$ to $\chi(\FF)$.

\section{Group-theoretic lemmas}
\label{S:group-theoretic lemmas}

Given any prime $\ell$, 
let $\ell'\colonequals \ell$ if $\ell \ne 2$, 
and $\ell'\colonequals 4$ if $\ell=2$.

\begin{lemma}[cf.~\cite{Minkowski1887}*{\S1}]
\label{L:Minkowski}
Let $\ell$ be a prime.
Let $G$ be a group acting through a finite quotient on
 a finite-rank free $\Z$-module or $\Z_\ell$-module $\Lambda$.
If $G$ acts trivially on $\Lambda/\ell' \Lambda$,
then $G$ acts trivially on $\Lambda$.
\end{lemma}

\begin{proof}
Let $n \colonequals \rk \Lambda$.
Write $\ell' \equalscolon \ell^s$.
For $r \ge s$, let $U_r \colonequals 1 + \ell^r M_n(\Z_\ell)$.
It suffices to show that there are no non-identity elements 
of finite order in the kernel $U_s$ of
$\GL_n(\Z_\ell) \to \GL_n(\Z_\ell/\ell' \Z_\ell)$.
In fact, for $r \ge s$ the binomial theorem shows that
$1+A \in U_r \setminus U_{r+1}$ implies $(1+A)^\ell \in U_{r+1} \setminus U_{r+2}$,
so by induction any non-identity $1+A \in U_s$ 
has infinitely many distinct powers,
and cannot be of finite order.
\end{proof}

\begin{lemma}
\label{L:cohomology}
Let a topological group $G$ act continuously on 
a finite-rank free $\Z_\ell$-module $\Lambda$.
Let $r\colonequals \rk \Lambda^G$.
Then the following hold.
\begin{enumerate}[\upshape (a)]
\item\label{I:H1 torsion is finite} 
The continuous cohomology group $\HH^1(G,\Lambda)[\ell^\infty]$ is finite.
\item\label{I:growth rate} $\#(\Lambda/\ell^n \Lambda)^G = O(\ell^{rn})$ as $n \to \infty$.
\end{enumerate}
\end{lemma}

\begin{proof}
For each $n$, taking continuous group cohomology of 
$0 \to \Lambda \stackrel{\ell^n}\to \Lambda \to \Lambda/\ell^n \Lambda \to 0$
yields
\begin{equation}
  \label{E:short exact sequence}
	0 \to \frac{\Lambda^G}{\ell^n(\Lambda^G)} 
	\to \left( \frac{\Lambda}{\ell^n \Lambda} \right)^G  
	\to \HH^1(G,\Lambda)[\ell^n]
	\to 0.
\end{equation}
\begin{enumerate}[\upshape (a)]
\item 
By~\eqref{E:short exact sequence} for $n=1$, 
the group $\HH^1(G,\Lambda)[\ell]$ is finite.
So if $\HH^1(G,\Lambda)[\ell^\infty]$ is infinite,
it contains a copy of $\Q_\ell/\Z_\ell$,
contradicting the $Y=0$ case of \cite{Tate1976}*{Proposition~2.1}.
\item 
In~\eqref{E:short exact sequence},
the group on the left has size $\ell^{rn}$,
and the group on the right has size $O(1)$
as $n \to \infty$, by~\eqref{I:H1 torsion is finite}.
\qedhere
\end{enumerate}
\end{proof}

\section{Upper bound on the rank of the group of Tate classes}
\label{S:stuff}

\begin{setup}
\label{Setup}
Let $k$ be a finitely generated field. 
Let $G\colonequals G_k$.
Let $X$ be a nice variety over $k$.
Let $d\colonequals \dim X$.
Fix $p \in \{0,1,\ldots,d\}$.
For each $m \in \Z_{>0}$ with $\kappa \nmid m$, 
define $T_m \colonequals \HH^{2p}(\Xsep,(\Z/m\Z)(p))$.
Fix a prime $\ell \ne \kappa$.
Define $T \colonequals \HH^{2p}(\Xsep,\Z_\ell(p))$,
and $V \colonequals \HH^{2p}(\Xsep,\Q_\ell(p))$.
\end{setup}

An element of $V$ is called a \defi{Tate class}
if it is fixed by a (finite-index) open subgroup of $G$.
Let $V^{\Tate} \le V$ be the $\Q_\ell$-subspace of Tate classes.
Let $M$ be the $\Z_\ell$-submodule of 
elements of $T$ mapping to Tate classes in $V$.
Let $r \colonequals \rk M = \dim V^{\Tate}$.

\begin{lemma}
\label{L:Kummer}
For each $i,n \in \Z_{\ge 0}$, there is an exact sequence
\[
	0 \to \frac{ \HH^i(\Xsep,\Z_\ell(p)) }{ \ell^n \HH^i(\Xsep,\Z_\ell(p)) }
	\to \HH^i(\Xsep,(\Z/\ell^n \Z)(p))
	\to \HH^{i+1}(\Xsep,\Z_\ell(p))[\ell^n]
	\to 0.
\]
\end{lemma}

\begin{proof}
Use \cite{MilneEtaleCohomology1980}*{Lemma~V.1.11} 
to take cohomology of 
\[
	0 \to \Z_\ell(p) \stackrel{\ell^n}\to \Z_\ell(p) 
		\to (\Z/\ell^n \Z)(p) \to 0.\qedhere
\]
\end{proof}

\begin{corollary}
\label{C:Kummer}
For each $n \ge 0$, there is an exact sequence
\[
	0 \to \frac{T}{\ell^n {T}} 
	\to T_{\ell^n}
	\to \HH^{2p+1}(\Xsep,\Z_\ell(p))[\ell^n]
	\to 0.
\]
\end{corollary}

\begin{proof}
Take $i=2p$ in Lemma~\ref{L:Kummer}.
\end{proof}

\begin{corollary}
\label{C:saturated Kummer}
For each $n \ge 0$, there is a canonical injection 
$M/\ell^n M \injects T_{\ell^n}$.
\end{corollary}

\begin{proof}
Since $M$ is saturated in $T$,
we have an injection $M/\ell^n M \injects T/\ell^n T$.
Compose with the first map in Corollary~\ref{C:Kummer}.
\end{proof}

\begin{lemma}
\label{L:size of W_n^G}
Let $t \in \Z_{\ge 0}$ be such that $\ell^t T_\tors =0$.
Assume that $G$ acts trivially on $T_{\ell'}$.
\begin{enumerate}[\upshape (a)]
\item\label{I:size 1} 
For any $n \geq t$,
we have $\# T_{\ell^n}^G \ge \ell^{r(n-t)}$.
\item\label{I:size 2} 
We have $\# T_{\ell^n}^G = O(\ell^{rn})$ 
as $n \to \infty$.
\item\label{I:rank min} We have 
\[
	r = \min \left\{ \left\lfloor \frac{\log \# T_{\ell^n}^G}{\log \ell^{n-t}} \right\rfloor : n > t \right\}.
\]
\end{enumerate}
\end{lemma}

\begin{proof}
By Corollary~\ref{C:saturated Kummer}, $G$ acts trivially on $M/\ell'M$,
and hence also on $M/\ell M$ and $\widetilde{M}/\ell' \widetilde{M}$.
The $G$-orbit of each element of $\widetilde{M}$ is finite
by definition of Tate class,
and $\widetilde{M}$ is finitely generated as a $\Z_\ell$-module,
so $G$ acts through a finite quotient on $\widetilde{M}$.
By Lemma~\ref{L:Minkowski}, $G$ acts trivially on $\widetilde{M}$.
\begin{enumerate}[\upshape (a)]
\item 
Multiplication by $\ell^t$ on $M$ kills $M_{\tors}$,
so it factors as $M \to \widetilde{M} \twoheadrightarrow \ell^t M$. 
Hence $G$ acts trivially on $\ell^t M$, so for $n\geq t$, 
the quotient $\ell^t M/\ell^n M$ is contained in $(M/\ell^n M)^G$. 
By Corollary~\ref{C:saturated Kummer}, we deduce the inequality
$\# T_{\ell^n}^G \ge \# (M/\ell^n M)^G \ge \# (\ell^t M/\ell^n M) \ge \ell^{r(n-t)}$.
\item
By definition of $M$, 
we have $\widetilde{T}^G \subseteq \widetilde{M} = \widetilde{M}^G \subseteq \widetilde{T}^G$,
so $\rk \widetilde{T}^G = r$.
Dividing the first two terms in Corollary~\ref{C:Kummer} 
by the images of $T_{\tors}$ yields
\[
0 \to \frac{\widetilde{T}}{\ell^n \widetilde{T}}
	\to \frac{T_{\ell^n}}{I_n}
	\to \HH^{2p+1}(\Xsep,\Z_\ell(p))[\ell^n]
	\to 0,
\]
where $I_n$ is the image of 
$T_{\tors}$
in $T_{\ell^n}$.
This implies the second inequality in
\[
	\# T_{\ell^n}^G 
	\le \# I_n^G 
	    \cdot \# \left( \frac{T_{\ell^n}}{I_n} \right)^G 
	\le \# I_n^G 
	    \cdot \# \left(\frac{\widetilde{T}}{\ell^n \widetilde{T}}\right)^G
            \cdot \# \left(\HH^{2p+1}(\Xsep,\Z_\ell(p))[\ell^n] \right)^G.
\]
Since $\HH^i(\Xsep,\Z_\ell(p))$ is a finitely generated $\Z_\ell$-module
for each $i$,
the first and third factors on the right are $O(1)$.
On the other hand, 
Lemma~\ref{L:cohomology}\eqref{I:growth rate} yields 
$\#(\widetilde{T}/\ell^n \widetilde{T})^G = O(\ell^{rn})$.
Multiplying shows that $\# T_{\ell^n}^G = O(\ell^{rn})$.
\item The statement follows by combining the previous items.
\qedhere
\end{enumerate}
\end{proof}



\section{Cycles under field extension}\label{S:field extension}

In this section, assume Setup~\ref{Setup}.

\begin{proposition}
\label{P:Num X}
\hfill
\begin{enumerate}[\upshape (a)]
\item\label{I:Num X injects} 
For any extension $L$ of $k$,
the natural map $\Num^p X \to \Num^p X_L$ is injective.
\item\label{I:Num X has finite index} 
The image of $\Num^p X \to \Num^p \Xbar$ 
is a finite-index subgroup of $(\Num^p \Xbar)^G$.
\item\label{I:Num X_sep vs Num Xbar}
If $\kappa>0$, the index of $\Num^p \Xsep$ in $\Num^p \Xbar$ 
is finite and equal to a power of $\kappa$.
\end{enumerate}
The same three statements hold for $\NS$ instead of $\Num^p$.
\end{proposition}

\begin{proof}\hfill
  \begin{enumerate}[\upshape (a)]
  \item 
If $z \in \calZ^p(X)$ has intersection number $0$ with all 
$p$-cycles on $X_L$,
then in particular it has intersection number $0$ with all 
$p$-cycles on $X$.
\item 
Suppose that $[z] \in (\Num^p \Xbar)^G$, where $z \in \calZ^p(\Xbar)$.
Then $z$ comes from some $z_L \in \calZ^p(X_L)$ for some finite 
extension $L$ of $k$.
Let $n \colonequals [L:k]$.
Then $n[z] = \tr_{L/k} [z]$ comes from $\tr_{L/k} z_L \in \calZ^p(X)$.
Hence the cokernel of $\Num^p X \to (\Num^p \Xbar)^G$ is torsion,
but it is also finitely generated, so it is finite.
\item 
We may assume that $k=\ksep$.
Then $G=\{1\}$, so~\eqref{I:Num X has finite index}
implies that $\Num^p \Xsep$ is of finite index in $\Num^p \Xbar$.
Moreover, in the proof of~\eqref{I:Num X has finite index},
$[L:k]$ is always a power of $\kappa$,
so the index is a power of $\kappa$.
  \end{enumerate}
Statement \eqref{I:Num X injects} for $\NS$
follows from the fact that 
the formation of $\PIC^0_{X/k}$ respects 
field extension~\cite{Kleiman2005}*{Proposition~9.5.3}.
The proofs of \eqref{I:Num X has finite index} 
and~\eqref{I:Num X_sep vs Num Xbar}
for $\NS$ are the same as for $\Num^p$.
\end{proof}

\begin{proposition}
\label{P:NS over finite field}
If $k$ is finite, then the natural homomorphisms
$\Pic X \to (\Pic \Xsep)^G$ and $\NS X \to (\NS \Xsep)^G$ are isomorphisms.
\end{proposition}

\begin{proof}
That $\Pic X \to (\Pic \Xsep)^G$ is an isomorphism
follows from the Hochschild--Serre spectral sequence for \'etale cohomology
and the vanishing of the Brauer group of $k$.
Lang's theorem~\cite{Lang1956} implies $\HH^1(k,\Pic^0 \Xsep)=0$,
so taking Galois cohomology of
\[
	0 \to \Pic^0 \Xsep \to \Pic \Xsep \to \NS \Xsep \to 0
\]
shows that the homomorphism 
$\Pic X = (\Pic \Xsep)^G \to (\NS \Xsep)^G$ is surjective.
On the other hand, its image is $\NS X$.
\end{proof}

\section{Hypotheses and conjectures}
\label{S:hypotheses and conjectures}

Our computability results rely on the ability to compute \'etale
cohomology with finite coefficients.
Some of the results are conditional also on the Tate conjecture
and related conjectures.
We now formulate these hypotheses precisely,
so that they can be referred to in our main theorems.

\subsection{Explicit representation of objects}
\label{S:explicit}

To specify an ideal in a polynomial ring over $\Z$ 
in finitely many indeterminates,
we give a finite list of generators.
To specify a finitely generated $\Z$-algebra $A$,
we give an ideal $I$ in a polynomial ring $R$ as above
such that $A$ is isomorphic to $R/I$.
To specify a finitely generated field $k$, 
we give a finitely generated $\Z$-algebra $A$ that is a domain
such that $k$ is isomorphic to $\Frac A$.
To specify a continuous $G_k$-action on a finitely generated abelian group $A$,
we give a finite Galois extension $k'$ of $k$ 
together with an action of $\Gal(k'/k)$ on $A$
such that there exists a $k$-embedding $k' \injects \ksep$
such that the original $G_k$-action is the composition
$G_k \surjects \Gal(k'/k) \to \Aut A$. 
To specify a $G_k$-action on finitely many finitely generated abelian groups,
we use the same $k'$ for all of them.
To specify a projective variety $X$, 
we give its homogeneous ideal for a particular embedding of $X$
in some projective space.
To specify a codimension~$p$ cycle on $X$, 
we give an explicit integer combination of codimension~$p$ integral
subvarieties of $X$.

\begin{definition}
\label{D:map on cycles}
Given $k$, $X$, and $p$ as in Setup~\ref{Setup},
to compute a $G_k$-module homomorphism $f$ from $\calZ^p(\Xsep)$ 
to an (abstract) finitely generated $G_k$-module $A$
means to compute
\begin{itemize}
\item a finite Galois extension $k'$ of $k$,
\item an explicit finitely generated $\Gal(k'/k)$-module $A'$, and
\item an algorithm that takes as input a finite separable extension $L$ of $k'$
and an element of $\calZ^p(X_L)$ and returns an element of $A'$,
\end{itemize}
such that there exists a $k$-embedding $k' \injects \ksep$
and an isomorphism $A' \isomto A$
such that the composition
$\calZ^p(X_L) \to A' \isomto A$
factors as $\calZ^p(X_L) \to \calZ^p(\Xsep) \stackrel{f} \to A$
for some (or equivalently, every) $k'$-embedding $L \injects \ksep$.
\end{definition}

\begin{remark}
\label{R:Z^tau}
A similar definition can be made for $G_k$-module homomorphisms
defined only on a $G_k$-submodule of $\calZ^p(\Xsep)$.
\end{remark}

\begin{remark}
\label{R:k in C}
If $k$ is a finitely generated field of characteristic~$0$,
we can explicitly identify finite extensions of $k$ with subfields of $\C$
consisting of computable numbers as follows.
(To say that $z \in \C$ is computable 
means that there is an algorithm that given $n \in \Z_{\ge 1}$
returns an element $\alpha \in \Q(i)$ such that $|z-\alpha| < 1/n$.)
Let $t_1,\ldots,t_n$ be a transcendence basis for $k$ over $\Q$.
Embed $\Q(t_1,\ldots,t_n)$ in $\C$ by mapping $t_j$ to $\exp(2^{1/j})$;
these are algebraically independent over $\Q$ 
by the Lindemann--Weierstrass theorem.
As needed, embed finite extensions of $\Q(t_1,\ldots,t_n)$ 
(starting with $k$) into $\C$ by 
writing down the minimal polynomial of each new field generator
over the subfield generated so far,
together with an approximation to an appropriate root in $\C$
good enough to distinguish it from the other roots.
\end{remark}

Remark~\ref{R:k in C} will be useful in relating \'etale cohomology
over $\kbar$ to singular cohomology over~$\C$.

\subsection{Computability of  \'etale cohomology}

\begin{hypothesis}[Cohomology is computable]
\label{H:compute etale 2}
There is an algorithm that takes as input $(k,X,\ell)$ as in Setup~\ref{Setup} 
and $i,n \in \Z_{\ge 0}$,
and returns a finite $G_k$-module isomorphic to $\HH^i(\Xsep,\Z/\ell^n\Z)$.
\end{hypothesis}

\begin{remark}
\label{R:compute Tate twist}
Hypothesis \ref{H:compute etale 2} implies also that we can compute 
the Tate twist 
$\HH^i(\Xsep,(\Z/\ell^n\Z)(p))\isom \HH^{i}(\Xsep,\Z/\ell^n\Z)(p)$ 
for any $p \in \Z$.
\end{remark}

We will prove Hypothesis~\ref{H:compute etale 2} for $k$ of characteristic~$0$
(Theorem~\ref{T:compute etale in char 0}).
In arbitrary characteristic, we show only that we can 
``approximate $\HH^i(\Xsep,\Z/\ell^n\Z)$ from below''
(Proposition~\ref{P:lower bound on Hyp groups}),
but as mentioned in the introduction, a proof of 
Hypothesis~\ref{H:compute etale 2} in full has been announced 
\cite{Madore-Orgogozo-preprint}*{Th\'eor\`eme~0.9}.

Following a suggestion of Lenny Taelman,
we use \'etale \Cech\ cocycles.
By \cite{Artin1971}*{Corollary 4.2}, 
every element of $\HH^i(\Xsep,\Z/\ell^n\Z)$
can be represented by a \Cech\ cocycle for some \'etale cover.
Any \'etale cover $\calU = (U_j \to \Xsep)_{j\in J}$
may be refined by one for which $J$ is finite
and the morphisms $U_j \to \Xsep$ are of finite presentation;
from now on, we assume that all \'etale covers satisfy these finiteness
conditions.
Then we can enumerate all \'etale \Cech\ cochains.

Fix a projective embedding of $X$.
Choose an \'etale \Cech\ cocycle representing the class
of $\OO_{\Xsep}(1)$ in $\HH^1(\Xsep,\G_m)$.
Using the Kummer sequence
\[
	0 \to \mu_{\ell^n} \to \G_m \to \G_m \to 0
\]
compute its coboundary:
this is a cocycle representing the class of a hyperplane section in 
$\HH^2(\Xsep,\Z/\ell^n\Z)$ (we ignore the Tate twist for now).
Compute its $d$-fold cup product in 
$\HH^{2d}(\Xsep,\Z/\ell^n) \isom \Z/\ell^n\Z$;
this represents $D$ times the class of a point,
where $D$ is the degree of $X$.
If $\ell \nmid D$, we can multiply by the inverse of $(D \bmod \ell)$
to obtain the class of a point.
In general, let $\ell^m$ be the highest power of $\ell$ dividing $D$;
repeat the construction above to obtain a cocycle $\eta_D$ representing
$D$ times the class of a point in 
$\HH^{2d}(\Xsep,\Z/\ell^{m+n}\Z) \isom \Z/\ell^{m+n}\Z$.
Search for another cocycle $\eta_1$ in the same group
such that $D \eta_1 - \eta_D$ is the coboundary of another cochain
on some refinement.
Eventually $\eta_1$ will be found, and reducing its values modulo $\ell^n$
yields a cocycle representing the class of a point in 
$\HH^{2d}(\Xsep,\Z/\ell^n\Z)$.

\begin{lemma}
\label{L:etale cocycle is 0}
There is an algorithm that takes as input $(k,X,\ell)$ 
as in Setup~\ref{Setup} and $i,n \in \Z_{\ge 0}$
and two \'etale \Cech\ cocycles
representing elements of $\HH^i(\Xsep,\Z/\ell^n)$,
and decides whether their classes are equal.
\end{lemma}

\begin{proof}
We can subtract the cocycles, so it suffices to test whether 
a cocycle $\eta$ represents $0$.
By day, search for a cochain on some refinement whose coboundary is $\eta$.
By night, search for a cocycle $\eta'$ 
representing a class in $\HH^{2d-i}(\Xsep,\Z/\ell^n\Z)$,
an integer $j \in \{1,2,\ldots,\ell^n-1\}$,
and a cochain whose coboundary differs from $\eta \cup \eta'$
by $j$ times the class of a point in $\HH^{2d}(\Xsep,\Z/\ell^n\Z)$
(see \cite{Liu2002}*{p.~194, Exercise~2.17} for an explicit formula
for the cup product).
The search by day terminates if the class of $\eta$ is $0$,
and the search by night terminates if the class of $\eta$ is nonzero,
by Poincar\'e duality \cite{SGA4.5}*{p.~71, Th\'eor\`eme~3.1}.
\end{proof}

\begin{proposition}
\label{P:lower bound on Hyp groups}
There is an algorithm that takes as input $(k,X,\ell)$ 
as in Setup~\ref{Setup} and $i,n \in \Z_{\ge 0}$
such that,
when left running forever, 
it prints out an infinite sequence $\Lambda_0 \subset \Lambda_1 \subset \ldots$
of finite $G_k$-modules 
that stabilizes at a $G_k$-module isomorphic to $\HH^i(\Xsep,\Z/\ell^n\Z)$.
\end{proposition}

\begin{proof}
By enumerating \v{C}ech cocycles, 
we represent more and more classes inside $\HH^i(\Xsep,\Z/\ell^n\Z)$.
At any moment, we may construct the $G_k$-module structure of 
the finite subgroup generated by the classes found so far 
and their Galois conjugates,
by using Lemma~\ref{L:etale cocycle is 0}
to test which $\Z/\ell^n\Z$-combinations of them are $0$.
Eventually this $G_k$-module is the whole of $\HH^i(\Xsep,\Z/\ell^n\Z)$
(even if we do not yet have a way to detect when this has happened).
\end{proof}

\begin{proposition}
\label{P:singular cohomology}
There is an algorithm that takes as input $(k,X,\ell)$ 
as in Setup~\ref{Setup} and $i,n \in \Z_{\ge 0}$,
where $k$ is of characteristic~$0$,
and computes a finite abelian group
isomorphic to the singular cohomology group $\HH^i(X(\C),\Z/\ell^n\Z)$
for some embedding $k \injects \C$ as in Remark~\ref{R:k in C}.
Similarly, one can compute a finitely generated abelian group
isomorphic to $\HH^i(X(\C),\Z)$.
\end{proposition}

\begin{proof}
One approach is to 
embed $X$ in some $\PP^n_k$ and compose $X(\C) \to \PP^n(\C)$ 
with the Mannoury embedding~\cite{Mannoury1900}
\begin{align*}
	\PP^n(\C) &\injects \R^{(n+1)^2} \\
	(z_0 : \cdots : z_n) &\mapsto 
		\left( \frac{z_i \bar{z}_j}{\sum_k z_k \bar{z}_k} : 
				0 \le i,j \le n \right)
\end{align*}
to identify $X(\C)$ with a semialgebraic subset of Euclidean space,
and then to apply \cite{Basu-Pollack-Roy2006}*{Remark~11.19(b) and the 
results it refers to} 
to compute a finite triangulation of $X(\C)$,
which yields the cohomology groups with coefficients in $\Z$ or $\Z/\ell^n\Z$.
For an alternative approach, see~\cite{Simpson2008}*{Section~2.5}.
\end{proof}

\begin{theorem}
\label{T:compute etale in char 0}
Hypothesis~\ref{H:compute etale 2} restricted to characteristic~$0$
is true.
\end{theorem}

\begin{proof}
Identify $k$ with a subfield of $\C$ as in Remark~\ref{R:k in C}.
By the standard comparison theorem,
the \'etale cohomology group $\HH^i(\Xsep,\Z/\ell^n\Z)$
is isomorphic to the singular cohomology group $\HH^i(X(\C),\Z/\ell^n\Z)$.
Use Proposition~\ref{P:singular cohomology} to compute
the size of the latter.
Run the algorithm in Proposition~\ref{P:lower bound on Hyp groups}
and stop once $\#\Lambda_j$ equals this integer.
Then $\Lambda_j \isom \HH^i(\Xsep,\Z/\ell^n\Z)$.
\end{proof}

\begin{corollary}
\label{C:compute etale by lifting}
Hypothesis~\ref{H:compute etale 2} restricted to 
varieties in positive characteristic
that lift to characteristic~$0$ is true.
\end{corollary}

\begin{proof}
If $X$ lifts to a nice variety $\calX$ in characteristic $0$,
then we can search for a suitable $\calX$ until we find one,
and then compute the size of $\HH^i(\calX^{\textup{sep}},\Z/\ell^n\Z)$, 
which is isomorphic to the desired group $\HH^i(\Xsep,\Z/\ell^n\Z)$.
Then run the algorithm in Proposition~\ref{P:lower bound on Hyp groups}
as before.
\end{proof}

\begin{remark}
\label{R:Taelman}
Our approach to Theorem~\ref{T:compute etale in char 0} above
was partially inspired by an alternative approach communicated to us
by Lenny Taelman.
His idea, in place of Proposition~\ref{P:lower bound on Hyp groups},
was to enumerate \'etale \Cech\ cocycles
and compute their images under a comparison isomorphism
\[
	\HH^i(\Xsep,\Z/\ell^n\Z) 
	\to \HH^i(X(\C),\Z/\ell^n\Z)
\]
explicitly (this assumes that given an \'etale morphism $U \to \Xsep$
one can compute compatible triangulations of $U(\C)$ and $X(\C)$).
Eventually a set of cocycles mapping bijectively
onto $\HH^i(X(\C),\Z/\ell^n\Z)$ will be found.
The Galois action could then be computed
by searching for coboundaries representing the difference
of each Galois conjugate of each cocycle 
with some other cocycle in the set.
\end{remark}

\subsection{The Tate conjecture}

See~\cite{Tate1994} for a survey of the relationships between the
following two conjectures and many others.

\begin{conjectureT}[Tate conjecture]
\label{C:Tate}
Assume Setup~\ref{Setup}.
The cycle class homomorphism
\[
	\calZ^p(\Xsep) \tensor \Q_\ell \to V^{\Tate}
\]
is surjective.
\end{conjectureT}

\begin{conjectureE}[Numerical equivalence equals homological equivalence]
\label{C:Num=Hom}
Assume Setup~\ref{Setup}.
An element of $\calZ^p(\Xsep)$ is numerically equivalent to $0$
if and only if its class in $V$ is $0$.
\end{conjectureE}

\begin{remark}
\label{R:E^1}
Conjecture $\rmE^1(X,\ell)$ holds (see \cite{Tate1994}*{p.~78}).
\end{remark}

Given $(k,X,p,\ell)$ as in Setup~\ref{Setup}, with $k$ finite,
let $V_\mu$ be the largest $G$-invariant subspace of $V$
on which all eigenvalues of the Frobenius are roots of unity.
We have $V^{\Tate} \le V_\mu$.

\begin{proposition}
\label{P:consequences}
Fix $X$, $p$, and $\ell$, and assume Conjecture $\rmE^p(X,\ell)$.
Then the following integers are equal:
\begin{enumerate}[\upshape (a)]
\item\label{I:Z-rank of Num} 
the $\Z$-rank of the $G_k$-module $\Num^p \Xsep$, 
\item\label{I:Z-rank of im Z} 
the $\Z$-rank of the image of $\calZ^p(\Xsep)$ in $V$, and
\item\label{I:Q_l-dim of im Z} 
the $\Q_\ell$-dimension of the image of $\calZ^p(\Xsep) \tensor \Q_\ell$ in $V$.
\end{enumerate}
The integer in \eqref{I:Q_l-dim of im Z} is less than or equal to 
the following equal integers,
\begin{enumerate}[\upshape (a)]
\item[\upshape (d)]\label{I:Q_l-dim of V^Tate}
the $\Q_\ell$-dimension of $V^{\Tate}$ and
\item[\upshape (e)]\label{I:Z_l-rank of M}
the $\Z_\ell$-rank of the $G_k$-module $M$ of Section~\ref{S:stuff},
\end{enumerate}
which, if $k$ is finite, are less than or equal to 
\begin{enumerate}[\upshape (a)]
\item[\upshape (f)]\label{I:Q_l-dim of V_mu}
the $\Q_\ell$-dimension of $V_\mu$.
\end{enumerate}
If moreover, $\rmT^p(X,\ell)$ holds,
then all the integers (including~(f) if $k$ is finite) 
are equal.
Conversely, if (c) equals~(d), then $\rmT^p(X,\ell)$ holds.
\end{proposition}

\begin{proof}
The only nontrivial statements are 
\begin{itemize}
\item 
the equality of \eqref{I:Z-rank of im Z} and~\eqref{I:Q_l-dim of im Z},
which is~\cite{Tate1994}*{Lemma~2.5}, and
\item 
the fact that $\rmT^p(X,\ell)$ and $\rmE^p(X,\ell)$ for $k$ finite
together imply the equality of (d) and~(f);
this follows from \cite{Tate1994}*{Theorem~2.9, (b)$\Rightarrow$(c)}.\qedhere
\end{itemize}
\end{proof}

\section{Algorithms}\label{S:algorithms}

\subsection{Computing rank and torsion of \'etale cohomology}\label{betti}

\begin{proposition}\label{P:zeta}
There is an algorithm that takes as input a nice variety $X$ over $\F_q$,
and returns its zeta function 
\[
	Z_X(T) \colonequals 
	\exp \left( \sum_{n=1}^\infty 
			\frac{\#X(\F_{q^n})}{n} T^n \right)
	\in \Q(T).
\]
\end{proposition}

\begin{proof}
{}From \cite{Katz2001}*{Corollary of Theorem~3},
we obtain an upper bound $B$ on 
the sum of the $\ell$-adic Betti numbers $b_i(X)$.
Then $Z_X(T)$ is a rational function of degree at most $B$.
Compute $\#X(\F_{q^n})$ for $n \in \{1,2,\ldots,2B\}$;
this determines the mod $T^{2B+1}$ Taylor expansion of $Z_X(T)$,
which is enough to determine $Z_X(T)$.
\end{proof}

\begin{proposition}\label{P:betti}
There is an algorithm that 
takes as input a finitely generated field $k$ and a nice variety $X$ over $k$,
and returns $b_0(X), \dots, b_{2\dim X}(X)$. 
\end{proposition}

\begin{proof}
First assume that $k=\F_q$.
Using Proposition~\ref{P:zeta}, we compute 
the zeta function $Z_X(T)$.
For each $i$, the Betti number $b_i(X)$ 
equals the number of complex poles of $Z_X(T)^{(-1)^i}$ with absolute value $q^{-i/2}$, counted with multiplicity;
this can be read off from the Newton polygon of the numerator or denominator of $Z_X(T)$.

In the general case, we spread out $X$
to a smooth projective scheme $\calX$ 
over a finitely presented $\Z$-algebra $R = \Z[x_1,\ldots,x_n]/(f_1,\ldots,f_m)$.
Search for a finite field $\F$ and a point $a \in \F^n$
satisfying $f_1(a)=\cdots=f_m(a)=0$;
eventually we will succeed; then $\F$ is an explicit $R$-algebra.
Set $\calX_\F = \calX \times_R \F$.
Standard specialization theorems (e.g., \cite{SGA4.5}*{V,~Th\'eor\`eme~3.1})
imply that $b_i(X)=b_i(\calX_\F)$ for all $i$, so we reduce to the case of the previous paragraph.
\end{proof}

The following statement and proof were suggested by Olivier Wittenberg.

\begin{proposition}
\label{P:torsion computable}
Assume Hypothesis~\ref{H:compute etale 2}. 
There is an algorithm that takes as input $(k,X,\ell)$ 
as in Setup~\ref{Setup} and an integer $i$ 
and returns a finite group that is isomorphic to $\HH^i(\Xsep,\Z_\ell)_\tors$.
\end{proposition}

\begin{proof}
For each $j$, let $\HH^j \colonequals \HH^j(\Xsep,\Z_\ell)$.
For integers $j,n$ with $n \ge 0$,
let $a_{j,n} \colonequals \# \HH^j[\ell^n]$ and 
$b_j \colonequals b_j(X)=\dim_{\Q_\ell}(\HH^j \otimes_{\Z_\ell} \Q_{\ell})$.
Since $\HH^j_\tors$ is finite, 
$\# \HH^j_\tors/\ell^n \HH^j_\tors = \# \HH^j_\tors[\ell^n] = a_{j,n}$, so 
$\#\HH^j/\ell^n\HH^j = \ell^{nb_j} \cdot a_{j,n}$. 
From Lemma \ref{L:Kummer} we find 
\begin{equation}\label{E:ajn}
\# \HH^j(\Xsep,\Z/\ell^n\Z) = \#\left(\HH^j/\ell^n\HH^j\right) \cdot \#\left( \HH^{j+1}[\ell^n]\right) 
   = \ell^{nb_j} \cdot a_{j,n} \cdot a_{j+1,n}.
\end{equation}
The left side is computable by 
Hypothesis~\ref{H:compute etale 2},
and $b_j$ is computable by Proposition~\ref{P:betti}.
Since $a_{j,n}=1$ for $j<0$ and for $j>2\dim X$, 
for any given $n$, we can use \eqref{E:ajn} to compute $a_{j,n}$ 
for all $j$, by ascending or descending induction.
Compute
$$
1=a_{i,0} \leq a_{i,1} \leq a_{i,2} \leq \dots \leq a_{i,N} \leq a_{i,N+1}
$$ 
until $a_{i,N} = a_{i,N+1}$. 
Then $\HH^i_\tors$ has exponent $\ell^N$ and $\HH^i_\tors$ is isomorphic 
to $\Directsum_{n=1}^N (\Z/\ell^n\Z)^{r_n}$ with $r_n$ such that 
$\ell^{r_n} a_{i,n-1}a_{i,n+1}= a_{i,n}^2$. 
\end{proof}

\begin{remark}
The proof of Proposition~\ref{P:torsion computable} did not require
the full strength of Hypothesis~\ref{H:compute etale 2}:
computability of the group $\HH^j(\Xsep,\Z/\ell^n\Z)$ 
for all $j<i$ or for all $j \ge i$ would have sufficed.
\end{remark}

\begin{remark}
\label{R:lifts to char 0}
If $k$ is of characteristic~$0$ (or $X$ lifts to characteristic~$0$), 
then combining Theorem~\ref{T:compute etale in char 0}
(or Corollary~\ref{C:compute etale by lifting})
with Proposition~\ref{P:torsion computable}
lets us compute the group $\HH^i(\Xsep,\Z_\ell)_\tors$ unconditionally.
\end{remark}

\subsection{Computing \texorpdfstring{$\Num^p \Xsep$}{Num p Xsep}}

Throughout this section, we assume Setup~\ref{Setup}.

\begin{lemma}
\label{L:computing intersection number}
There is an algorithm that takes as input 
$k$, $p$, $X$, and cycles $y \in \calZ^p(X)$ and $z \in \calZ^{d-p}(X)$,
and returns the intersection number $y.z$.
\end{lemma}

\begin{proof}
First, if $y$ and $z$ are integral cycles intersecting transversely,
use Gr\"obner bases to compute the degree of their intersection.
If $y$ and $z$ are arbitrary cycles whose supports intersect transversely,
use bilinearity to reduce to the previous sentence.
In general, search for a rational equivalence between $y$
and another $p$-cycle $y'$ such that the supports of $y'$ and $z$
intersect transversely;
eventually $y'$ will be found; 
then apply the previous sentence to compute $y'.z$.
\end{proof}

\begin{remark}
It should be possible to make the algorithm in the proof of 
Lemma~\ref{L:computing intersection number} much more efficient, 
by following a proof of Chow's moving lemma 
instead of finding $y'$ by brute force enumeration.
\end{remark}

\begin{remark}
Alternatively, if $y$ and $z$ are integral cycles of complementary dimension 
that do not necessarily intersect properly,
their structure sheaves $\OO_y$ and $\OO_z$ 
admit resolutions $\FF^\bullet$ and $\GG^\bullet$ 
(complexes of locally free $\OO_X$-modules),
and then 
\begin{equation}
\label{E:intersection using resolution}
	y.z = \sum_{i,j \ge 0} (-1)^{i+j} \chi(\FF^i \tensor \GG^j);
\end{equation}
this should lead to another algorithm.
(Formula~\eqref{E:intersection using resolution} can be explained as follows:
Replace $y$ and $z$ by rationally equivalent cycles $y'$ and $z'$ 
that intersect transversely; then, in $K(X)$,
\begin{align*}
	\cl(\OO_{y'} \tensor \OO_{z'})
	&= \cl(\OO_{y'}) \cl(\OO_{z'})  \quad \textup{(by \cite{SGA6}*{p.~49, Proposition~2.7})} \\
	&= \cl(\OO_y) \cl(\OO_z) \quad \textup{(by \cite{SGA6}*{p.~59, Corollaire~1}, using $\dim y + \dim z = d$)} \\
	&= \sum_{i \ge 0} (-1)^i \cl(\FF^i) \sum_{j \ge 0} (-1)^j \cl(\GG^j) \\
	&= \sum_{i,j \ge 0} (-1)^{i+j} \cl(\FF^i \tensor \GG^j) \quad \textup{(by \cite{SGA6}*{p.~49, (2.15~bis)}).}
\end{align*}
Since $\OO_{y'} \tensor \OO_{z'}$ is a direct sum of skyscraper sheaves,
applying $\chi \colon K(X) \to \Z$ yields $y'.z'$ on the left,
which equals $y.z$.)

Similarly, one could prove the simpler but asymmetric formula
$y.z = \sum_{i \ge 0} (-1)^i \chi(\FF^i \tensor \OO_z)$.
\end{remark}

The following lemma describes a decision problem 
for which we do not have an algorithm that always terminates,
but only a \emph{one-sided} test,
i.e., an algorithm that halts if the answer is YES,
but runs forever without reaching a conclusion if the answer is NO.

\begin{lemma}
\label{L:independent in Num}
There is an algorithm that takes as input 
$k$, $p$, $X$, a finite extension $L$ of $k$, 
and a finite list of cycles $z_1,\ldots,z_s \in \calZ^p(X_L)$,
and halts if and only if the images of $z_1,\ldots,z_s$ in $\Num^p \Xbar$
are $\Z$-independent.
\end{lemma}

\begin{proof}
Enumerate $s$-tuples $(y_1,\ldots,y_s)$ of elements of $\calZ^{d-p}(X_{L'})$
as $L'$ ranges over finite extensions of $L$.
As each $s$-tuple is computed, compute also 
the intersection numbers $y_i.z_j \in \Z$
and halt if $\det(y_i.z_j) \ne 0$.
\end{proof}

\begin{remark}
If $p=1$ and $d=2$, and $h$ is any ample divisor on $X$,
and $z$ is an integer combination of the $h$ and the $z_i$,
then the Hodge index theorem shows that 
the numerical class of $z$ is $0$
if and only if $z.h=0$ and $z.z_j=0$ for all $j$;
thus the independence in Lemma~\ref{L:independent in Num}
can be tested by calculating intersection numbers of
already-known divisors without having to search for elements $y_i$.
If $p=1$ and $d>2$, and we assume the Hodge standard conjecture 
\cite{Grothendieck-standard-conjectures}*{Section~4, Conjecture~$\operatorname{Hdg}(X)$},
then the numerical class of $z$ is $0$
if and only if $z.h^{d-1}=0$ and $z.z_j.h^{d-2}=0$ for all $j$;
thus again the search for the $y_j$ is unnecessary, conjecturally.
A similar argument applies for higher~$p$
if we assume not only the Hodge standard conjecture
but also that an algebraic cycle $\lambda$ 
as in the Lefschetz standard conjecture 
\cite{Grothendieck-standard-conjectures}*{Section~3, Conjecture~$B(X)$}
can be found algorithmically
so that one can compute the primitive decomposition of $z$ and the $z_j$ 
before computing intersection numbers.
\end{remark}

\begin{remark}
\label{R:independence in Num^p Xsep}
In Lemma~\ref{L:independent in Num}, if $L$ is separable over $k$,
then it would be the same to ask for independence in $\Num^p \Xsep$,
by Proposition~\ref{P:Num X}\eqref{I:Num X injects}.
\end{remark}

\begin{corollary}
\label{C:lower bounds for Num}
There is an algorithm that takes as input 
$k$, $p$, and $X$, and that when left running forever,
prints out an infinite sequence of nonnegative integers whose maximum 
equals $\rk \Num^p \Xsep$.
\end{corollary}

\begin{proof}
Enumerate finite $s$-tuples $(z_1,\ldots,z_s)$ of elements 
of $\calZ^p(X_L)$
for all $s \ge 0$ and all finite separable extensions $L$ of $k$,
and run the algorithm of Lemma~\ref{L:independent in Num}
(using Remark~\ref{R:independence in Num^p Xsep})
on all of them in parallel, devoting a fraction $2^{-i}$ of the 
algorithm's time to the $i^{\tH}$ process.
Each time one of the processes halts, print its value of $s$.
\end{proof}

\begin{theorem}[Computing $\Num^p \Xsep$]
\label{T:Num}
\hfill
\begin{enumerate}[\upshape (a)]

\item \label{I:rank of Num}
Assume Hypothesis~\ref{H:compute etale 2}.
Then there is an algorithm that takes as input 
$(k,X,p,\ell)$ as in Setup~\ref{Setup}
such that, assuming $\rmE^p(X,\ell)$,
\begin{itemize}
\item the algorithm terminates if and only if $\rmT^p(X,\ell)$ holds, and
\item if the algorithm terminates, it returns $\rk \Num^p \Xsep$.
\end{itemize}

\item\label{I:unconditional Num}
There is an unconditional algorithm 
that takes $k$, $p$, $X$, and a nonnegative integer $\rho$ as input,
and computes the following assuming that $\rho = \rk \Num^p \Xsep$:
\begin{enumerate}[\upshape (i)]
\item\label{I:Num N}
a finitely generated torsion-free $G_k$-module $N$ 
having a $G_k$-equivariant injection $\Num^p \Xsep \injects N$
with finite cokernel,
\item\label{I:map to N}
the composition $\calZ^p(\Xsep) \to \Num^p \Xsep \injects N$
in the sense of Definition~\ref{D:map on cycles}, and
\item the rank of $\Num^p X$.
\end{enumerate}

\end{enumerate}
\end{theorem}

\begin{proof}
\hfill
\begin{enumerate}[\upshape (a)]
\item 
Let $\ell'$ be as in Section~\ref{S:group-theoretic lemmas}.
Use Hypothesis~\ref{H:compute etale 2} to compute $T_{\ell'}$.
Replace $k$ by a finite Galois extension to assume that $G_k$
acts trivially on $T_{\ell'}$.
Let $M$ and $r$ be as in Section~\ref{S:stuff}.

Use the algorithm of Proposition~\ref{P:torsion computable} to compute 
an integer $t$ such that $\ell^t T_\tors=0$. 
By day, use Hypothesis~\ref{H:compute etale 2} to compute the groups $T_{\ell^n}$
for $n=t+1,t+2,\ldots$,
and the upper bounds $\lfloor \log \#T_{\ell^n}^G / \log \ell^{n-t} \rfloor$
on $r$ given by Lemma~\ref{L:size of W_n^G}\eqref{I:rank min}.
By night, compute lower bounds on $\rk \Num^p \Xsep$ 
as in Corollary~\ref{C:lower bounds for Num}.
Stop if the bounds ever match, 
which happens if and only if equality holds in the inequality
$\rk \Num^p \Xsep \le r$,
which by Proposition~\ref{P:consequences}
happens if and only if $\rmT^p(X,\ell)$ holds.
In this case, we have computed $\rk \Num^p \Xsep$.
\item
  \begin{enumerate}[\upshape (i)]
  \item 
Search for a finite Galois extension $k'$ of $k$, 
for $p$-cycles $y_1,\ldots,y_s$, 
and for codimension~$p$ cycles $z_1,\ldots,z_t$ over $k'$
until the intersection matrix $(y_i.z_j)$ has rank $\rho$.
The assumption $\rho=\rk \Num^p \Xsep$ guarantees that such 
$k'$, $y_i$, $z_j$ will be found eventually.
Let $Y$ be the free abelian group with basis equal to the 
set consisting of the $y_i$ and their Galois conjugates,
so $Y$ is a $G_k$-module.
The intersection pairing defines a homomorphism 
$\phi \colon \Num^p \Xsep \to \Hom_\Z(Y,\Z)$
whose image has rank equal to $\rho=\rk \Num^p \Xsep$.
Since $\Num^p \Xsep$ is torsion-free, $\phi$ is injective.
Compute the saturation $N$ of the $\Z$-span of $\phi(z_1),\ldots,\phi(z_s)$
in $\Hom_\Z(Y,\Z)$.
Because of its rank, $N$ equals the saturation of $\phi(\Num^p \Xsep)$.
Thus $N$ is a finitely generated torsion-free $G_k$-module
containing a finite-index $G_k$-submodule $\phi(\Num^p \Xsep)$
isomorphic to $\Num^p \Xsep$.
\item
Given $z \in \calZ^p(X_L)$ for some finite separable extension $L$ of $k'$,
computing its intersection number with each basis element of $Y$
yields the image of $z$ in $N$.
\item 
Because of Proposition~\ref{P:Num X}\eqref{I:Num X has finite index},
$\rk \Num^p X = \rk N^{G_k}$, which is computable.\qedhere
  \end{enumerate}
\end{enumerate}
\end{proof}

\begin{remark}
If we can bound the exponent of $T_{\tors}=\HH^{2p}(\Xsep,\Z_\ell)_\tors$
without using Proposition~\ref{P:torsion computable},
then Theorem~\ref{T:Num}\eqref{I:rank of Num}
requires Hypothesis~\ref{H:compute etale 2} only for $i=2p$.
In particular, this applies if $\Char k=0$
or if $\Char k > 0$ and $X$ lifts to characteristic~$0$,
by Remark~\ref{R:lifts to char 0}.
Actually, if $\Char k=0$, 
we do not need Hypothesis~\ref{H:compute etale 2} at all,
because Theorem~\ref{T:compute etale in char 0} says that it is true!
\end{remark}

\begin{remark}
\label{R:Num Xbar}
The analogue of Theorem~\ref{T:Num} with $\Xsep$ replaced by $\Xbar$
also holds, as we now explain.
By Proposition~\ref{P:Num X}\eqref{I:Num X_sep vs Num Xbar},
$\Num^p \Xsep$ is of finite index in $\Num^p \Xbar$,
so in the proof of Theorem~\ref{T:Num}\eqref{I:unconditional Num}(i),
the homomorphism $\phi$ extends to a $G_K$-equivariant injective homomorphism 
$\overline{\phi} \colon \Num^p \Xbar \to \Hom_\Z(Y,\Z)$.
Because of finite index, the image of $\overline{\phi}$ is contained in 
the $N$ defined there.
The cokernel of $\Num^p \Xbar \to N$ is finite.
\end{remark}

\begin{remark}
\label{R:intersection pairing}
For each $p \in \{0,1,\ldots,d\}$, 
let $N_p$ be the $N$ in Theorem~\ref{T:Num}\eqref{I:unconditional Num}(i),
and define $Q_p \colonequals N_p \tensor \Q$.
Then for any $p,q \in \Z_{\ge 0}$ with $p+q \le d$,
we can compute a bilinear pairing $Q_p \times Q_q \to Q_{p+q}$
that corresponds to the intersection pairing:
indeed, each $Q_p$ is spanned by classes of cycles,
whose intersections in the Chow ring can be computed
by an argument similar to that used to prove 
Lemma~\ref{L:computing intersection number}.
\end{remark}

\subsection{Checking algebraic equivalence of divisors}

\begin{lemma}
\label{L:algebraically equivalent to 0}
There is an algorithm that takes as input 
$k$, $X$, a finite extension $L$ of $k$, 
and an element $z \in \calZ^1(X_L)$,
and halts if and only if $z$ is algebraically equivalent to $0$.
\end{lemma}

\begin{proof}
Enumerate all possible descriptions of an algebraic family of divisors
on $X_L$ with a pair of $L$-points of the base (it is easy to check when
such a description is valid), and check for each
whether the difference of the cycles corresponding to the two points
equals $z$.
\end{proof}

\begin{lemma}
\label{L:algebraically torsion p=1}
There is an algorithm that takes as input 
$k$, $X$, a finite extension $L$ of $k$
and $z \in \calZ^1(X_L)$,
and decides whether $z$ lies in $\calZ^1(X_L)^\tau$,
i.e., whether the N\'eron--Severi class of $z$ is torsion,
i.e., whether $z$ is numerically equivalent to $0$.
\end{lemma}

\begin{proof}
By day, search for a positive integer $n$
and a family of divisors showing that $nz$ 
is algebraically equivalent to $0$.
By night, run the algorithm of
Lemma~\ref{L:independent in Num} for $s=1$,
which halts if and only if the image of $z$ in $\Num^1 \Xbar$
is nonzero, i.e., if and only if $z \notin \calZ^1(X_L)^\tau$.
One of these processes will halt.
\end{proof}

\subsection{Computing the N\'eron--Severi group}

In this section, $k$ is an \emph{arbitrary} field.

\begin{lemma}
\label{L:Keeler}
\hfill
\begin{enumerate}[\upshape (a)]
\item \label{I:Keeler a}
Let $X$ be a nice $k$-variety.
There exists a divisor $B \in \calZ^1_{X/k}$
such that for any ample divisor $D$, the class of $D+B$ is very ample.
\item \label{I:Keeler b}
There is an algorithm that takes as input a finitely generated field $k$ 
and a $k$-variety $X$
and computes a $B$ as in~\eqref{I:Keeler a}.
\end{enumerate}
\end{lemma}

\begin{proof}
Let $K$ be a canonical divisor on $X$ (this is computable if $k$ is
finitely generated).
Let $A$ be a very ample divisor on $X$ (e.g., embed $X$ in some projective
space, and choose a hyperplane section).
By~\cite{Keeler2008}*{Theorem~1.1(2)}, 
$B \colonequals K + (\dim X + 1) A$ has the required property.
\end{proof}

Given an effective Cartier divisor of $X$,
we have an associated closed subscheme $Y \subseteq X$.
Call a closed subscheme $Y \subseteq X$ a divisor 
if it arises this way.
When we speak of the Hilbert polynomial of an effective Cartier divisor on 
a closed subscheme $X$ of $\PP^n$,
we mean the Hilbert polynomial of the associated closed subscheme of $X$.

\begin{lemma}
\label{L:compute Hilbert}
There is an algorithm that takes as input a finitely generated field $k$, 
a closed subscheme $X \subseteq \PP^n_k$, 
and an effective divisor $D \subset X$,
and computes the Hilbert polynomial of $D$.
\end{lemma}

\begin{proof}
This is evident already from~\cite{Hermann1926}*{Satz~2},
which can be applied repeatedly to construct 
a minimal free resolution of $\OO_D$.
\end{proof}

Let $\Hilb X = \Union_P \Hilb_P X$ denote the Hilbert scheme of $X$,
where $P$ ranges over polynomials in $\Q[t]$.

\begin{lemma}
\label{L:Gotzmann}
There is an algorithm that takes as input a finitely generated field $k$, 
a closed subscheme $X \subseteq \PP^n_k$, and a polynomial $P \in \Q[t]$,
and computes the universal family $\calY \to \Hilb_P X$.
\end{lemma}

\begin{proof}
This is a consequence of work of Gotzmann.
Let $S = \Directsum_{d \ge 0} S_d \colonequals k[x_0,\ldots,x_n]$,
so $\Proj S = \PP^n_k$.
Given $d,r \in \Z_{\ge 0}$,
let $\Gr_r(S_d)$ be the Grassmannian parametrizing
$r$-dimensional subspaces of the $k$-vector space $S_d$.
Then \cite{Gotzmann1978}*{\S3} 
(see also \cite{Iarrobino-Kanev1999}*{Theorem~C.29 and Corollary~C.30})
specifies $d_0 \in \Z_{\ge 0}$ such that for $d \ge d_0$,
one can compute $r \in \Z_{\ge 0}$ and a closed subscheme 
$W \subseteq \Gr_r(S_d)$
such that $W \isom \Hilb_P \PP^n$;
under this isomorphism a subspace $V \subseteq S_d$ 
corresponds to the subscheme defined by 
the ideal $I_V$ generated by the polynomials in $V$.
Moreover, $I_V$ and its saturation have the same $d^{\tH}$ graded part 
(see~\cite{Iarrobino-Kanev1999}*{Corollary~C.18}).

Let $f_1,\ldots,f_m$ be generators of a homogeneous ideal defining $X$.
Choose $d \in \Z$ such that $d \ge d_0$ and $d \ge \deg f_i$ for all $i$.
Let $g_1,\ldots,g_M$ be all the polynomials obtained
by multiplying each $f_i$ by all monomials of degree $d-\deg f_i$.
By the saturation statement above, 
$\Proj(S/I_V) \subseteq X$ if and only if $g_j \in V$ for all $j$.
This lets us construct $\Hilb_P X$ as an explicit closed subscheme
of $\Hilb_P \PP^n$.
Now $\Hilb_P X$ is known as an explicit subscheme of the Grassmannian,
so we have explicit equations also for the universal family over it.
\end{proof}

\begin{lemma}
\label{L:EffDiv}
Let $X$ be a nice $k$-variety.
There exists an open and closed subscheme $\EffDiv_X \subseteq \Hilb X$
such that for any field extension $L \supseteq k$ 
and any $s \in (\Hilb X)(L)$, 
the closed subscheme of $X_L$ corresponding to $s$ is a divisor on $X_L$ 
if and only if $s \in \EffDiv_X(L)$.
\end{lemma}

\begin{proof}
See \cite{Bosch-Lutkebohmert-Raynaud1990}*{p.~215} for the 
definition of the functor $\EffDiv_X$ 
(denoted there by $\Div_{X/S}$ for $S=\Spec k$)
and its representability by an open subscheme of $\Hilb X$.
To see that it is also closed,
we apply the valuative criterion for properness to the inclusion 
$\EffDiv_X \to \Hilb X$:
if a $k$-scheme $S$ is the spectrum of a discrete valuation ring
and $Z$ is a closed subscheme of $X \times S$ that is flat over $S$
and the generic fiber $Z_\eta$ of $Z \to S$ is a divisor,
then $Z$ equals the closure of $Z_\eta$ in $X \times S$, 
which is an effective Weil divisor on $X \times S$
and hence a relative effective Cartier divisor 
since $X \times S$ is regular.
\end{proof}

The existence of the scheme $\EffDiv_X$ in Lemma~\ref{L:EffDiv}
immediately implies the following.

\begin{corollary}
\label{C:divisors under field extension}
Let $X$ be a nice $k$-variety.
Let $Y$ be a closed subscheme of $X$.
Let $L$ be a field extension of $k$.
Then $Y$ is a divisor on $X$ if and only if $Y_L$ is a divisor on $X_L$.
\end{corollary}

\begin{remark}
Corollary~\ref{C:divisors under field extension}
holds more generally for any finite-type $k$-scheme $X$,
as follows from fpqc descent applied to the ideal sheaf of $Y_L \subseteq X_L$.
\end{remark}

\begin{lemma}
\label{L:test for divisor}
There is an algorithm that takes as input 
a finitely generated field $k$, 
a smooth $k$-variety $X$,
and a closed subscheme $Y \subseteq X$,
and decides whether $Y$ is a divisor in $X$.
\end{lemma}

\begin{proof}
By \cite{EGA-IV.IV}*{Proposition~21.7.2} 
or \cite{Eisenbud1995}*{Theorem~11.8a.},
$Y$ is a divisor if and only if
all associated primes of $Y$ are of codimension~$1$ in $X$.
So choose an affine cover $(X_i)$ of $X$, 
compute the associated primes of the ideal of $Y \intersect X_i$ in $X_i$
for each $i$ (the first algorithm was given in~\cite{Hermann1926}),
and check whether they all have codimension~$1$ in $X_i$
(a modern method for computing dimension 
uses that the Hilbert polynomial of an ideal
equals the Hilbert polynomial of an associated initial ideal,
which can be computed from a Gr\"obner basis).
\end{proof}

\begin{lemma}
\label{L:connected components}
Let $\pi \colon H \to P$ be a proper morphism
of schemes of finite type over a field $k$.
Suppose that the fibers of $\pi$ are connected (in particular, nonempty).
Then $\pi$ induces a bijection on connected components.
\end{lemma}

\begin{proof}
Let $H_1,\ldots,H_n$ be the connected components of $H$.
Let $P_i\colonequals \pi(H_i)$, so $P_i$ is connected.
Since $\pi$ is proper, the $P_i$ are closed.
Since the fibers of $\pi$ are connected, the $P_i$ are disjoint.
Since the fibers are nonempty, $\Union P_i = P$.
Since the $P_i$ are finite in number, they are open too,
so they are the connected components of $P$.
\end{proof}

Let $\pi \colon \EffDiv_X \to \PIC_{X/k}$ be the proper morphism sending
a divisor to its class.
If $\PIC_{X/k}^c$ is a finite union of connected components of $\PIC_{X/k}$
and $L$ is a field extension of $k$,
let $\Pic^c X_L$ be the set of classes in $\Pic X_L$ 
such that the corresponding point of $\left(\PIC_{X/k}\right)_L$ 
lies in $\left(\PIC^c_{X/k}\right)_L$,
and let $\NS^c X_L$ be the image of $\Pic^c X_L$ in $\NS X_L$.

\begin{lemma}
\label{L:no functors}
\hfill
\begin{enumerate}[\upshape (a)]
\item \label{I:nice bijection}
Let $X$ be a nice $k$-variety.
Let $\PIC_{X/k}^c$ be any finite union of connected components of $\PIC_{X/k}$.
Assume the following:
\begin{equation}
  \label{E:Keeler condition}  
\begin{aligned}
 & \textup{For every field extension $L \supseteq k$, 
every divisor on $X_L$ with class in $\Pic^c X_L$} \\
 & \textup{is linearly equivalent to an effective divisor.}
\end{aligned}
\end{equation}
Let $H \colonequals \pi^{-1}(\PIC_{X/k}^c)$.
Then $\pi \colon H(L) \to \Pic X_L$ induces a bijection
\begin{equation}
  \label{E:connected components}
	\{\textup{connected components of $H_L$ that contain an $L$-point}\}
	\To \NS^c X_L.
\end{equation}
\item \label{I:A exists}
For any $\PIC^c_{X/k}$ as in~\eqref{I:nice bijection},
there is a divisor $F$ on $X$ such that the translate $F + \PIC_{X/k}^c$ 
satisfies \eqref{E:Keeler condition}.
\item \label{I:D and e}
There is an algorithm that takes as input
a finitely generated field $k$, a nice $k$-variety $X$,
a divisor $D \in \calZ^1(X)$, and a positive integer $e$,
and computes the following
for $\PIC^c_{X/k}$ defined as the (possibly empty) 
union of components of $\PIC_{X/k}$ 
corresponding to classes of divisors $E$ over $\kbar$
such that $eE$ is numerically equivalent to $D$:
\begin{enumerate}[\upshape (i)]
\item 
a divisor $F$ as in~\eqref{I:A exists} for $\PIC^c_{X/k}$,
\item 
the variety $H$ in~\eqref{I:nice bijection} for $F+\PIC^c_{X/k}$,
\item 
the universal family $Y \to H$ of divisors corresponding to points of $H$,
\item 
a finite separable extension $k'$ of $k$ 
and a finite subset $\calS \subseteq \calZ^1(X_{k'})$
such that there exists a $k$-homomorphism $k' \injects \ksep$
such that the composition $\calZ^1(X_{k'}) \to \calZ^1(\Xsep) \to \NS \Xsep$
restricts to a bijection $\calS \to \NS^c \Xsep$.
\end{enumerate}
\end{enumerate}
\end{lemma}

\begin{proof}\hfill
  \begin{enumerate}[\upshape (a)]
  \item 
Taking $L=\kbar$ in~\eqref{E:Keeler condition}
shows that $H \stackrel{\pi}\to \PIC_{X/k}^c$ is surjective.
The fibers of $\pi \colon H(L) \to \Pic^c X_L$ are linear systems,
and are nonempty by~\eqref{E:Keeler condition},
so the reduced geometric fibers of $\pi \colon H \to \PIC^c_{X/k}$
are projective spaces.
In particular, $\pi_L \colon H_L \to \left(\PIC^c_{X/k}\right)_L$ 
has connected fibers, so by Lemma~\ref{L:connected components},
it induces a bijection on connected components.
Under this bijection, the connected components of $H_L$ 
that contain an $L$-point
map to the connected components of $\left(\PIC^c_{X/k}\right)_L$ 
containing the class of a divisor over $L$.
The set of the latter components is $\NS^c X_L$.
\item 
Let $A$ be an ample divisor on $X$.
For each of the finitely many geometric components $C$ of $\PIC^c_{X/k}$,
choose a divisor $D_C$ on $X_\kbar$ whose class lies in $C$,
and let $n_C \in \Z$ be such that $n_C A + D_C$ is ample.
Let $n=\max n_C$, so $n A + D_C$ is ample for all $C$.
Let $B$ be as in Lemma~\ref{L:Keeler}\eqref{I:Keeler a}.
Let $F=B+nA$.
If $L$ is a field extension of $k$ 
and $E$ is a divisor on $X_L$ with class in $\Pic^c X_L$,
let $C$ be the geometric component containing the class of $E_{\Lbar}$
(for some compatible choice of $\kbar \subseteq \Lbar$);
then $E$ is numerically equivalent to $D$, so $n A + E$ is ample too,
so $F + E = B + (nA + E)$ is very ample by choice of $B$,
so $F+E$ is linearly equivalent to an effective divisor.
\item 
Fix a projective embedding of $X$, 
and let $A$ be a hyperplane section.
  \begin{enumerate}[\upshape (i)]
  \item 
Let $n \in \Z_{>0}$ be such that $nA+D$ is ample.
(To compute such an $n$, try $n=1,2,\ldots$ until 
$|nA+D|$ determines a closed immersion.)
Compute $B$ as in Lemma~\ref{L:Keeler}\eqref{I:Keeler b}.
Let $F=B+nA$.
Suppose that $L$ is an extension of $k$ and $E$ is a divisor on $X_L$
such that $eE$ is numerically equivalent to $D$.
Then $e(nA+E)$ is numerically equivalent to $enA+D = (e-1)nA + (nA+D)$,
which is a positive combination of the ample divisors $A$ and $nA+D$,
so $nA+E$ is ample.
By choice of $B$, the divisor $F+E = B+(nA+E)$ is very ample
and hence linearly equivalent to an effective divisor.
\item
By the Riemann--Roch theorem, the Euler characteristic 
$\chi(F + sD + tA)$
is a polynomial $f(s,t)$ of total degree 
at most $d \colonequals \dim X$.
For any $s \in \Z$, we can compute $t \in \Z$ such that $F+sD+tA$
is linearly equivalent to an effective divisor, 
whose Hilbert polynomial can be computed by Lemma~\ref{L:compute Hilbert},
so the polynomial $\chi(F + sD + tA)$ can be found by interpolation.
Let $P(t) \colonequals f(1/e,t)$.
Compute the universal family $\calY \to \Hilb_P X$ 
as in Lemma~\ref{L:Gotzmann}.

Suppose that $E$ is such that $eE$ is numerically equivalent to $D$.
Then the polynomial $\chi(F+sE+tA)$ equals $f(s/e,t)$
since its values match whenever $e|s$.
In particular, $\chi(F+E+tA)=P(t)$;
i.e., $P(t)$ is the Hilbert polynomial of an effective divisor
linearly equivalent to $F+E$.
Thus the subscheme $H \subseteq \EffDiv_X \subseteq \Hilb X$
is contained in $\Hilb_P X$, which is a union of connected components 
of $\Hilb X$.
By definition, $H$ is a union of connected components of $\EffDiv_X$,
which by Lemma~\ref{L:EffDiv} is a union of connected components
of $\Hilb X$, so $H$ is a union of connected components of $\Hilb_P X$.
To compute $H$, 
compute the (finitely many) connected components of $\Hilb_P X$;
to check whether a component $C$ belongs to $H$, 
choose a point $h$ in $C$ over some extension of $k$,
apply Lemma~\ref{L:test for divisor} to $\calY_h$
to test whether $\calY_h$ is a divisor,
and if so, apply Lemma~\ref{L:algebraically torsion p=1}
to $e \calY_h - D$ 
to check whether $e \calY_h$ is numerically equivalent to $D$.
\item 
Compute $Y \to H$ as the part of $\calY \to \Hilb_P X$ above $H$.
\item 
Compute the connected components of $H_{k^\sep}$, 
which really means computing a finite separable extension $k'$
and the connected components of $H_{k'}$
such that these components are geometrically connected.
For each connected component $C$ of $H_{k'}$, 
use the algorithm of~\cite{Haran1988}
to decide whether it has a $k^\sep$-point,
and, if so, choose a $k'$-point $h$ of $C$,
enlarging $k'$ if necessary,
and take the fiber $Y_h$.
Let $\calS$ be the set of such divisors $Y_h$,
one for each component $C$ with a $\ksep$-point.
By~\eqref{I:nice bijection}, 
the map $\calS \to \NS \Xsep$ is a bijection onto $\NS^c \Xsep$.
\qedhere
  \end{enumerate}
  \end{enumerate}
\end{proof}

\begin{theorem}[Computing $(\NS \Xsep)_{\tors}$]
\label{T:NS_tors}
There is an algorithm that takes as input 
a finitely generated field $k$ and a nice $k$-variety $X$,
and computes the $G_k$-homomorphism
$\calZ^1(\Xsep)^{\tau} \to (\NS \Xsep)_\tors$
sending a divisor to its N\'eron--Severi class,
in the sense of Definition~\ref{D:map on cycles} and Remark~\ref{R:Z^tau}.
\end{theorem}

\begin{proof}
Apply Lemma~\ref{L:no functors}\eqref{I:D and e}
with $D=0$ and $e=1$
to obtain a finite Galois extension $k'$
and a subset $\calD \subseteq \calZ^1(X_{k'})$
mapping bijectively to $(\NS \Xsep)_\tors$.
For each pair $D_1,D_2 \in \calD$,
run Lemma~\ref{L:algebraically equivalent to 0}
in parallel on $D_1+D_2-D_3$ for all $D_3 \in \calD$
to find the unique $D_3$ algebraically equivalent to $D_1+D_2$;
this determines the group law on $\calD$.
Similarly compute the $G_k$-action.
Similarly, given a finite separable extension $L$ of $k'$
and $z \in \calZ^1(X_L)^\tau$, we can find the unique $D \in \calD$
algebraically equivalent to $z$.
\end{proof}

If $\calD \subseteq \calZ^1(X_\ksep)$,
let $(\NS \Xsep)^\calD$ be the saturation of the $G_k$-submodule
generated by the image of $\calD$ in $\NS \Xsep$,
and let $\calZ^1(\Xsep)^\calD$ 
be the set of divisors in $\calZ^1(\Xsep)$ 
whose algebraic equivalence class lies in $(\NS \Xsep)^\calD$.

\begin{theorem}[Computing $\NS \Xsep$]
\label{T:NS}
\hfill
\begin{enumerate}[\upshape (a)]
\item \label{I:saturate D}
Given a finitely generated field $k$, a nice $k$-variety $X$, 
a finite separable extension $L$ of $k$ in $\ksep$
and a finite subset $\calD \subseteq \calZ^1(X_L)$,
we can compute the $G_k$-homomorphism 
$\calZ^1(\Xsep)^\calD \to (\NS \Xsep)^\calD$
in the sense of Definition~\ref{D:map on cycles} and Remark~\ref{R:Z^tau}.
\item \label{I:compute the whole NS}
There is an algorithm that takes as input $k$ and $X$ as above 
and a nonnegative integer $\rho$,
and computes the $G_k$-homomorphism 
$\calZ^1(\Xsep) \to (\NS \Xsep)$
in the sense of Definition~\ref{D:map on cycles} and Remark~\ref{R:Z^tau}
assuming that $\rho = \rk \NS \Xsep$.
\end{enumerate}
\end{theorem}

\begin{remark}
\label{R:NS}
Assume Hypothesis~\ref{H:compute etale 2} and $\rmT^1(X,\ell)$.
(Conjecture~$\rmE^1(X,\ell)$ is proved.)
Then Theorem~\ref{T:Num}\eqref{I:rank of Num}  
lets us compute $\rk \NS \Xsep$,
so Theorem~\ref{T:NS}\eqref{I:compute the whole NS} 
lets us compute $\NS \Xsep$.
Recall also that Hypothesis~\ref{H:compute etale 2} 
is true when restricted to characteristic~$0$ 
(Theorem~\ref{T:compute etale in char 0})
or varieties that lift to characteristic~$0$
(Corollary~\ref{C:compute etale by lifting}).
\end{remark}

\begin{proof}[Proof of Theorem~\ref{T:NS}]
\hfill
\begin{enumerate}[\upshape (a)]
\item
Enlarge $L$ to assume that it is Galois over $k$,
and replace $\calD$ by the union of its $\Gal(L/k)$-conjugates.
There exist $D_1,\ldots,D_t \in \calD$ whose images in
$(\Num^1 \Xsep) \tensor \Q$ form a $\Q$-basis
for the image of the span of $\calD$.
Then there exist $1$-dimensional cycles $E_1,\ldots,E_t$ on $X_L$ 
such that $\det (D_i.E_j) \ne 0$ 
(the $E_i$ exist over a finite extension of $L$, but can be replaced
by their traces down to $L$),
and each $D \in \calD$ has a positive integer multiple
numerically equivalent to an element of the $\Z$-span of $\calD$.
Search for such $D_1,\ldots,D_t,E_1,\ldots,E_t$
and for numerical relations as above for each $D \in \calD$
(use Lemma~\ref{L:algebraically torsion p=1} to verify relations).
Let $e \colonequals \left| \det (D_i.E_j) \right|$.
Let $\Delta$ be the span of the image of $\calD$ in $\Num^1 \Xsep$.
Let $\Delta'$ be the saturation of $\Delta$ in $\Num^1 \Xsep$.
Then $(\Delta':\Delta)$ divides $e$.
For each coset of $e \Delta$ in $\Delta$,
choose a representative divisor $D$ in the $\Z$-span of $\calD$,
and check whether 
the set $\calS$ of Lemma~\ref{L:no functors}\eqref{I:D and e} is nonempty 
to decide whether the numerical equivalence class of $D$ is in $e \Delta'$;
if so, choose a divisor in $\calS$.
The classes of these new divisors, together with those of $D_1,\ldots,D_t$,
generate $\Delta'$.
Moreover, we know the integer relations between all of these,
so we can compute integer combinations $F_1,\ldots,F_t$ whose classes 
form a \emph{basis} for $\Delta'$.
Then 
\[
	(\NS \Xsep)^\calD \isom (\Z F_1 \directsum \cdots \directsum \Z F_t) \directsum (\NS \Xsep)_{\tors}
\]
as abelian groups, 
and $(\NS \Xsep)_{\tors}$ can be computed by Theorem~\ref{T:NS_tors}.

The homomorphism $\calZ^1(\Xsep)^\calD \to (\NS \Xsep)^\calD$
is computed as follows:
given any $z \in \calZ^1(\Xsep)^\calD$ 
(defined over some finite separable extension $L'$ of $L$ in $\ksep$),
compute an integer combination $F$ of the $F_i$ 
such that $F.E_j = z.E_j$ for all $j$,
and apply the homomorphism of Theorem~\ref{T:NS_tors}
to compute the class of $z-F$ in $(\NS \Xsep)_{\tors}$.

Applying this to all conjugates of our generators of 
$(\NS \Xsep)^\calD$ lets us compute the $G_k$-action on
our model of $(\NS \Xsep)^\calD$.
\item 
Assuming that $\rho=\rk \NS \Xsep$,
the algebraic equivalence classes 
of divisors $D_1,\ldots,D_\rho \in \calZ^1(\Xsep)$
form a $\Z$-basis for a free subgroup of finite index in $\NS \Xsep$
if and only if there exist $1$-cycles $E_1,\ldots,E_\rho$ on $\Xsep$
such that $\det (D_i.E_j) \ne 0$.
Search for a finite separable extension $L$ of $k$ in $\ksep$,
divisors $D_1,\ldots,D_\rho \in \calZ^1(X_L)$,
and $1$-cycles $E_1,\ldots,E_\rho$ on $X_L$
until such are found with $\det (D_i.E_j) \ne 0$.
Then apply \eqref{I:saturate D} to 
$\calD \colonequals \{D_1,\ldots,D_\rho\}$.\qedhere
\end{enumerate}
\end{proof}

\begin{remark}
\label{R:kbar}
Theorems \ref{T:NS_tors} and~\ref{T:NS} hold for $\Xbar$ instead of $\Xsep$:
the same proofs work,
except that we need an algorithm for deciding 
whether a variety has a $\kbar$-point;
fortunately,
this is even easier than deciding whether a variety has a $\ksep$-point!
\end{remark}

\subsection{An alternative approach over finite fields}
\label{S:alternative}

When $k$ is a finite field, 
we can compute $\rk \Num^p \Xsep$
without assuming Hypothesis~\ref{H:compute etale 2}, 
but still assuming $\rmT^p(X,\ell)$ and $\rmE^p(X,\ell)$.
The arguments in this section are mostly well-known.

The following is a variant of Theorem~\ref{T:Num}\eqref{I:rank of Num}.
Recall that for any $(k,X,p,\ell)$ as in Setup~\ref{Setup} with $k$ finite, 
we let $V_\mu$ denote the largest $G$-invariant subspace of 
$V=\HH^{2p}(\Xsep,\Q_\ell(p))$ on which all eigenvalues of the Frobenius are roots of unity.

\begin{theorem}
\label{T:Num variant}
\hfill
\begin{enumerate}[\upshape (a)]
\item \label{I:dim V_mu}
There is an algorithm A that 
takes as input $(k,X,p,\ell)$ as in Setup~\ref{Setup},
with $k$ a finite field $\F_q$,
and returns $\dim V_\mu$.
\item \label{I:algorithm B}
There is an algorithm B 
that takes as input $(k,X,p,\ell)$ as in Setup~\ref{Setup},
with $k$ a finite field $\F_q$,
such that, assuming $\rmE^p(X,\ell)$,
\begin{itemize}
\item algorithm B terminates on this input 
if and only if $\rmT^p(X,\ell)$ holds, and
\item if algorithm B terminates, it returns $\rk \Num^p \Xsep$.
\end{itemize}
\end{enumerate}
\end{theorem}

\begin{proof}
\hfill
\begin{enumerate}[\upshape (a)]
\item
By Proposition~\ref{P:zeta} there is an algorithm that computes the 
zeta function $Z_X(T)\in \Q(T)$ of $X$.
Then $\dim V_\mu$ is the number of complex poles $\lambda$ 
of $Z_X(T)$ such that $\lambda$ is a root of unity times $q^{-p}$,
counted with multiplicity.
\item 
Algorithm B
first runs algorithm A to compute $v_\mu \colonequals \dim V_\mu$,
and then runs the algorithm of Corollary~\ref{C:lower bounds for Num}
until it prints $v_\mu$,
in which case algorithm B returns $v_\mu$.
If $\rmT^p(X,\ell)$ and $\rmE^p(X,\ell)$ hold,
Proposition~\ref{P:consequences}
implies that $v_\mu$ equals $\rk \Num^p \Xsep$,
and the algorithm of Corollary~\ref{C:lower bounds for Num} 
eventually prints the latter,
so algorithm B terminates with the correct output.

Assume $\rmE^p(X,\ell)$.
Proposition~\ref{P:consequences} implies that 
$\rk \Num^p \Xsep \le v_\mu$ with equality if and only if 
$\rmT^p(X,\ell)$ holds.
So if algorithm B terminates, then $\rmT^p(X,\ell)$ holds.
\qedhere
\end{enumerate}
\end{proof}

\begin{corollary}
\label{C:NS over finite field under Tate}
There is an algorithm 
to compute $\NS \Xsep$ 
(in the same sense as Theorem~\ref{T:NS}\eqref{I:compute the whole NS})
and its subgroup $\NS X$
for any nice variety $X$ over a finite field
such that $\rmT^1(X,\ell)$ holds for some $\ell$.
\end{corollary}

\begin{proof}
Apply Theorem~\ref{T:Num variant}\eqref{I:algorithm B},
using that $\rmE^p(X,\ell)$ holds for $p=1$,
to obtain $\rk \NS \Xsep$.
Then Theorem~\ref{T:NS}\eqref{I:compute the whole NS} 
lets us compute the Galois module $\NS \Xsep$.
By Proposition~\ref{P:NS over finite field},
computing its $G_k$-invariant subgroup yields $\NS X$.
\end{proof}

\subsection{K3 surfaces}
\label{S:K3}

We now apply our results to K3 surfaces,
to improve upon the results of \cite{Charles2014} 
and \cite{Hassett-Kresch-Tschinkel2013}
mentioned in Section~\ref{S:previous approaches}.

\begin{theorem}
\label{T:unconditional NS for K3}
There is an unconditional algorithm
to compute the $G_k$-module $\NS \Xsep$ for any K3 surface $X$ 
over a finitely generated field $k$ of characteristic not~$2$.
We can also compute the group $(\NS \Xsep)^{G_k}$,
in which $\NS X$ has finite index.
If $k$ is finite, we can compute $\NS X$ itself.
\end{theorem}

\begin{proof}
By \cite{Deligne1981}, K3 surfaces lift to characteristic~$0$.
By \cite{Madapusi-preprint}*{Theorem~1}, 
$\rmT^1(X,\ell)$ holds for any K3 surface $X$ 
over a finitely generated field $k$ of characteristic not~$2$.
Hence Remark~\ref{R:NS} lets us compute the $G_k$-module $\NS \Xsep$.
{}From this we obtain $(\NS \Xsep)^{G_k}$.
By Proposition~\ref{P:Num X}, $\NS X$ is of finite index in $(\NS \Xsep)^{G_k}$.
If $k$ is finite, 
then $\NS X = (\NS \Xsep)^{G_k}$ by Proposition~\ref{P:NS over finite field}.
\end{proof}

\begin{remark}
For K3 surfaces $X$ over a finite field $k$ of characteristic not~$2$,
Corollary~\ref{C:NS over finite field under Tate}
yields another way to compute $\NS \Xsep$,
without lifting to characteristic~$0$,
but still using \cite{Madapusi-preprint}*{Theorem~1}.
\end{remark}

\section*{Acknowledgements} 

We thank 
Saugata Basu, 
Fran\c{c}ois Charles,
Bas Edixhoven, 
Robin Hartshorne, 
David Holmes, 
Moshe Jarden, 
J\'anos Koll\'ar,
Andrew Kresch, 
Martin Olsson, 
Lenny Taelman, 
Burt Totaro, 
David Vogan, 
Claire Voisin, 
Olivier Wittenberg,
and the referee
for helpful comments.
We thank the Banff International Research Station, 
the American Institute of Mathematics,
the Centre Interfacultaire Bernoulli,
and the Mathematisches Forschungsinstitut Oberwolfach
for their hospitality and support.

\end{document}